\documentclass[10pt]{amsart}

\usepackage{graphicx,color}
\usepackage[colorlinks=true, linkcolor=blue,citecolor=blue, bookmarks=true, pdfstartview=FitH]{hyperref}
\usepackage[driver=pdftex,margin=3cm,heightrounded=true,centering]{geometry}
\usepackage{latexsym}
\usepackage{cite}
\usepackage{amsmath,amssymb,amsthm,mathrsfs}

\allowdisplaybreaks[1]  

\newcommand{\mfld}{\mathcal{M}}   
\newcommand{\sym}{\mathcal{S}}    

\newcommand{\bS}{\mathbb{S}}
\newcommand{\bZ}{\mathbb{Z}}

\newcommand{\bH}{\mathbb{H}}
\newcommand{\bT}{\mathbb{T}}

\newcommand{\vep}{\varepsilon}
\newcommand{\pd}[2]{\frac{\partial #1}{\partial #2}}

\newcommand{\mo}{\mathcal{O}}
\newcommand{\md}{\mathcal{D}}
\newcommand{\mx}{\mathcal{X}}
\newcommand{\my}{\mathcal{Y}}

\newcommand{\inn}[2]{\left\langle #1, #2 \right\rangle}

\DeclareMathOperator{\Rm}{Rm}
\DeclareMathOperator{\Rc}{Rc}

\DeclareMathOperator{\scal}{Scal}

\DeclareMathOperator{\dvol}{dvol}

\newcommand{\bR}{\mathbb{R}}
\newcommand{\bC}{\mathbb{C}}
\newcommand{\bN}{\mathbb{N}}

\newcommand{\sO}{\mathcal{O}}

\newcommand{\gee}{\mathfrak{g}}
\newcommand{\gtil}{\widetilde{g}}
\newcommand{\nabtil}{\widetilde{\nabla}}
\newcommand{\Rtil}{\widetilde{\Rm}}
\newcommand{\Rctil}{\widetilde{\Rc}}

\newcommand{\ppt}{\frac{\partial}{\partial t}}

\newtheorem{theorem}{Theorem}
\newtheorem{prop}[theorem]{Proposition}

\newtheorem{corollary}[theorem]{Corollary}

\newcounter{mnotecount}[section]
\let\oldmarginpar\marginpar
\renewcommand\marginpar[1]{\-\oldmarginpar[\raggedleft\footnotesize #1]%
{\raggedright\footnotesize #1}}

\begin{document}

\title[Convergence Stability]{Convergence Stability for Ricci Flow}

\author[Bahuaud]{Eric Bahuaud}
\address{Department of Mathematics,
	Seattle University,
	Seattle, WA, 98122, USA}
\email{bahuaude(at)seattleu.edu}
\author[Guenther]{Christine Guenther}
\address{Department of Mathematics,
Pacific University,
Forest Grove, OR 97116, USA}
\email{guenther@pacific.edu}
\author[Isenberg]{James Isenberg}
\address{Department of Mathematics,
University of Oregon,
Eugene, OR 97403-1222 USA}
\email{isenberg@uoregon.edu}

\date{\today}

\maketitle

\begin{abstract} The principle of {\it convergence stability} for geometric flows is the combination of the continuous dependence of the flow on initial conditions, with the stability of fixed points. It implies that if the flow from an initial state $g_0$ exists for all time and converges to a stable fixed point, then the flows of solutions that start near $g_0$ also converge to fixed points. We show this in the case of the Ricci flow, carefully proving the continuous dependence on initial conditions.  Symmetry assumptions on initial geometries are often made to simplify geometric flow equations. As an application of our results, we extend known convergence results to open sets of these initial data, which contain geometries with no symmetries. 
\end{abstract}

\section{Introduction}

In the analysis of the long term behavior of geometric heat flows, it is often useful to consider initial geometries with symmetry, since those symmetries are preserved and so the nonlinear flow equations are greatly simplified \cite{CIJ, IJ, LS, MCF}. If, given a topology and appropriate notion of stability, the flow solution starting at some  metric $g_0$ does converge to a stable fixed-point geometry, does it follow that solutions starting at geometries nearby converge to a stable fixed-point geometry as well? Since geometric heat flows are generally weakly parabolic PDE systems, and since one expects such systems to have a well-posed initial value problem with solutions depending continuously on initial data, it is plausible that flow solutions starting near $g_0$ \emph{do} always converge.

In this work, we make these ideas precise, and apply them to the Ricci flow on compact manifolds.  Given a Riemannian manifold $(\mathcal{M},g),$ the Ricci flow is the weakly parabolic system 
\begin{equation}
\begin{cases}
\label{ricciflow}
\frac{\partial}{\partial t} g(t) &= -2 \Rc(t) \\
g(0) &= g_0, \\
\end{cases}
\end{equation}
where $\Rc$ is the Ricci curvature tensor. We work in the topology of little H\"older spaces of metrics as in  \cite{GIK}.  We define a geometric heat flow to be \emph{convergence stable} if for every initial geometry $g_0$ whose resulting solution converges to a stable fixed-point geometry, there is a neighbourhood of $g_0$ such that every solution starting in that neighbourhood also converges to a stable fixed-point geometry. 

The key to proving convergence stability for a given flow is the verification that the flow is a well-posed initial value problem, with well-posedness encompassing continuous dependence of the solutions on initial data in addition to local existence and uniqueness. The continuous dependence of the Ricci flow for metrics on compact manifolds has been broadly assumed, but we have been unable to find a complete proof of it in the literature (e.g. see \cite{MOF}). Since our results here crucially involve continuous dependence in a certain topology, in Section \ref{Sec:ContDepRF} we provide a careful proof of this property.  

A complete understanding in the compact case is necessary to a subsequent project in which we study the flow on non-compact manifolds, and we therefore include an amount of detail that additionally provides an accessible entry into this theory for geometric analysts. In the course of this exposition, we in particular prove a crucial resolvent bound whose proof is either omitted or incomplete in various treatments of stability for geometric flows, see Proposition \ref{resolvent-estimate}.

We now describe our results.  We endow the bundle of symmetric two tensors, $\sym$, with a topology given by H\"older norms.  In what follows, the spaces $h^{k,\alpha}(\sym)$, $k \in \bN_0, \alpha \in (0,1)$, denote the completion of the space of smooth symmetric two tensors with respect to the $(k,\alpha)$-H\"older norm (see the Appendix for a detailed definition of these spaces).  Given a metric $g_0$, we let $\tau(g_0) \in (0,\infty]$ denote the maximal time of existence of the Ricci flow starting at $g_0$.  We prove the following theorem: 

\

\noindent{\bf Theorem A.} (Continuous Dependence of the Ricci Flow). \label{thm:ctsdep} {\it Let $(\mfld,g_0)$ be a smooth compact Riemannian manifold. Let $g_0 (t)$ be the maximal solution of the Ricci flow \eqref{ricciflow} for $t \in [0,\tau(g_0))$, \  $\tau(g_0) \le \infty$. Choose $0 < \tau < \tau(g_0)$. Let $k \geq 2$.
There exist positive constants $r$ and $C$ depending only on $g_0$ and $\tau$ such that if $$ ||g_1 - g_0||_{h^{k+2,\alpha}} \le r,$$ then for the unique solution $g_1(t)$ of the Ricci flow starting at $g_1$, the maximal existence time satisfies 
$$\tau(g_1) \ge \tau,$$ and
\begin{equation*}
||g_1(t) - g_0(t) ||_{h^{k,\alpha}}  \le C || g_1 - g_0 ||_{h^{k+2,\alpha}} 
\end{equation*} 
for all $t \in [0,\tau].$}

\

The proof of this theorem involves the Ricci-DeTurck flow, to which parabolic PDE theory may be applied; our results give a semigroup proof of well-posedness.  Since we work in finite regularity spaces, the passage from the Ricci-DeTurck flow to the Ricci flow accounts for the drop in regularity in the theorem above.  Note that we are not asserting a drop in regularity of the solution $g_1(t)$ which is smooth after $t>0$, only the differentiability of the dependence estimate.  After giving some analytic background in Section \ref{Sec:Analytic},  in Section \ref{Sec:RDF}  we present a detailed exposition of the continuous dependence of the Ricci-DeTurck flow using analytic semigroup theory.  Working within the setting of little H\"older spaces, we can simplify various parts of the semigroup-based treatment of the well-posedness of parabolic equations; see,  e.g., \cite{Lun95}.  In Section \ref{Sec:ContDepRF}, we return to the Ricci flow and prove Theorem A. 

Having established the continuous dependence of the Ricci flow for metrics on compact manifolds, we prove convergence stability of the Ricci flow on closed $\bT^n$ and $\mathbb{H}^n$ in Section \ref{ConvStab}.  Denote the Ricci flow solution starting at $g_0$ by $g_0(t)$.  We say {$g_0(t)$ {\it converges to $g_{\infty}$ in $h^{2,\alpha}$} if $\tau(g_0)=\infty$ and
\[ \lim_{t \to \infty} \| g_0(t) - g_{\infty} \|_{h^{2,\alpha}} = 0. \]
Note that by results of Hamilton, this is sufficient to obtain convergence in $C^{\infty}(\mfld)$, \cite[Section 17]{Hamilton}.

We say $g_{\infty}$ is a {\it stable fixed point} if there is some $\vep > 0$ so that for every $g_0$ within the $\vep$-ball about $g_{\infty}$ in $h^{2,\alpha}$, the Ricci flow starting at $g_0$ converges to a fixed point.  D. Knopf and two of the authors proved that any flat metric on the torus $\bT^n$ is a stable fixed point for the Ricci flow in this sense, see \cite[Theorem 3.7]{GIK}.  In \cite{KnopfYoung}, the D. Knopf and A. Young prove that a compact hyperbolic metric is also a stable fixed point for a normalized Ricci flow in this sense. 

The statement of the $\bT^n$ result is as follows:

\

\noindent{\bf Theorem B.} (Convergence Stability (Flat)).  \label{thm:cstab-flat} {\it Let $(\mfld, g_0)$ be a smooth closed torus, and let $g_0(t)$ be a solution of the Ricci flow \eqref{ricciflow} starting at $g_0$ that converges to a flat metric $g_\infty$. There exists $r>0$ depending only on $g_0$  such that if $g_1 \in B_r(g_0) \subset h^{2,\alpha}(\sym)$, then the solution $g_1(t)$ with $g(0) = g_1$ of the Ricci flow exists for all time and converges to a flat metric.}

\

This theorem shows that the set of metrics which converge to a flat fixed point is \emph{open} in a certain topology of metrics on $\bT^n$.  We emphasize that the metrics contained in these sets need not have symmetries, nor are they generally small perturbations from a fixed flat metric. We apply this to the results of J. Lott and N. Sesum, who showed that for two families of geometries on $\bT^3$ defined by isometries, every Ricci flow solution which starts at a metric contained in one of these families must converge exponentially quickly to a flat metric on $\bT^3$ \cite{LS}. By Theorem B, we obtain open neighbourhoods of the set of these metrics such that Ricci flow solutions which start with initial geometries in either of  these neighbourhoods must converge to flat metrics. We discuss this and other applications in Section \ref{ConvStab}. 

We strongly expect that the notion of convergence stability holds for a wide range of other geometric heat flows.  We hope to return to these questions in a subsequent work.

In the Appendix we provide some of the  background definitions and results needed in this paper including the basic function spaces which arise and a few results from elliptic PDE theory. 

\textbf{Acknowledgments}{  The authors would like to thank Dan Knopf, Jack Lee, Rafe Mazzeo, and Haotian Wu for helpful conversations related to this work.  This work was supported by grants from the Simons Foundation (\#426628, E. Bahuaud and \#283083, C. Guenther). J. Isenberg was partially supported by the NSF grant DMS-1263431. This work was initiated at the 2015 BIRS workshop 
{\it Geometric Flows: Recent Developments and Applications} (15w5148).}

\section{Analytic background}
\label{Sec:Analytic}

In this section, we introduce some of the ideas and results from analytic semigroup theory that we need in order to prove our main results on continuous dependence and convergence stability for Ricci flow. An overview of stability techniques and results for geometric flows, and in particular maximal regularity theory,  can be found in Chapter 35 of \cite{RFV4}. An excellent reference for this theory in a general setting is provided by \cite{Lun95}. A careful exposition of the geometric ideas that we use, including how to define the relevant function spaces of tensors, as well as certain key facts from elliptic PDE theory, can be found in Appendix A. 

As motivation of the semigroup approach, we recall that in the analysis of systems of ordinary differential equations, for a constant $n$ by $n$ matrix $A$ over $\mathbb C$, and for a vector-valued function of time $x = x(t) \in \bC^n$, the solution to the initial value problem
\begin{equation*}
\begin{cases} 
x'(t) &= A x(t),  \\
x(0) &= x_0,\\
\end{cases} 
\end{equation*} 
exists for all time and can be written in terms of the matrix exponential in the following form: $x(t) = e^{t A} x_0$.  One readily verifies that the long-time behavior of solutions near equilibrium points depends on the eigenvalues of $A$.  We discuss an extension of this setting to general unbounded operators acting on  Banach spaces, as well as nonlinear generalizations.

Let $\mx$ be a Banach space, and let 
\[ A: \md(A) \subset \mx \longrightarrow \mx, \]
be a closed linear operator.  A linear operator  $A$ is {\it closed} if its graph is a closed subset of $\mx \times \mx$.  In that case $\md(A)$ with the graph norm, $\|x\|_{\md(A)} := \|x\| + \|Ax\|$, is a Banach space. The  \emph{spectrum} of $A$ with respect to the Banach space $\mx$ is defined to be the set
\[ \sigma_\mx(A) := \{ \lambda \in \bC: \lambda I - A \; \; \mbox{does not have a bounded inverse} \},  \]
while the  corresponding \emph{resolvent set} is defined as 
\[\rho_\mx(A) := \bC \smallsetminus \sigma_\mx(A).\]
For a chosen value  $\lambda$  in the resolvent set $\rho_\mx(A)$, the corresponding resolvent operator is the map
\begin{equation*}
 R_{\lambda} := (\lambda I - A)^{-1}: \mx \longrightarrow \md(A) \subset \mx. 
 \end{equation*}
We denote by $\mathcal{L}(\mx)$ the set of bounded linear operators on a normed space $\mx$, and we denote by 
$\mathcal{L}(\mx,\my)$ the set of bounded operators from $\mx$ to $\my$. A \emph{sector} is defined to be the set
\[ S_{\theta,\omega} = \{ \lambda \in \bC: \lambda \neq \omega, |\arg(\lambda -\omega)| < \theta\}.\]
The linear operator $A$ is \emph{sectorial} if there exist constants $\omega \in \bR$ and $\theta \in (\pi/2, \pi)$ and $M > 0$ such that 

\

\begin{enumerate}
\item[R1.] (The resolvent set contains a sector): 
the set $S_{\theta,\omega}$ is a subset of $\rho_\mx(A).$

\item [R2.] (Resolvent estimate): the inequality  
\begin{equation}
\label{ResEst}
\| R_{\lambda} \|_{\mathcal{L}(\mx)} \leq \frac{M}{|\lambda - \omega|}
\end{equation} 
holds for all 
$\lambda \in S_{\lambda,\omega}$.
\end{enumerate}

\

Now consider the Banach space valued initial value problem
\begin{equation}
\label{IVP}
u'(t) = A u(t), \; \; u(t_0) = u_0, 
\end{equation}
where $u_0 \in \md(A).$  If $A$ is sectorial, then for an appropriate contour $\gamma$ in the complex plane and for $t>0$, one may define the exponential of $tA$ by
\[ e^{tA} := \frac{1}{2\pi i} \int_{\gamma} e^{t \lambda} R_\lambda \; d\lambda, \]
and analyzing the mapping properties of $e^{tA} $, one obtains well-posedness for the linear IVP \eqref{IVP}. 

Sectoriality is important for nonlinear initial value problems as well. In particular, beginning with a continuous embedding of Banach spaces $\md \subset \mx$, we let $\sO \subset \md$ be an open subset, and we consider the evolution equation
\[ u'(t) = F(u(t)), \; \; u(t_0) = u_0, \]
where 
\[ F: \sO \longrightarrow \mx. \]

\
We require three conditions on the flow:

\begin{itemize} \label{hypothesis} 

\item[H0.] $F$ is continuous and Fr\'echet differentiable.

\item[H1.] For each $u \in \sO$, the Fr\'echet derivative $F_u$ of $F$ is sectorial in $\mx$, and its graph norm is equivalent to the norm of $\md$.

\item[H2.] For each $u \in \sO$, there exist positive constants $R$ and $L$ depending on $u$ such that the inequality 
\[ \| F_u(v) - F_u(w)\|_{\mathcal{L}(\md,\mx)} \leq L \| v - w\|_\md, \]
holds for all $v, w \in B_R(u) \subset \md$. 
\end{itemize} 

\

To state a well-posedness theorem for a nonlinear map $F$ which satisfies conditions H1 and H2, one needs to address a technicality which arises in the study of parabolic evolution systems: A key feature of parabolic systems is \emph{parabolic regularity}, which is the propensity of solutions of parabolic evolution solutions to become smooth  even if the initial data is significantly less regular. Hence, in the formulation of a well-posedness theorem, one must account for diminished regularity as $t$ approaches zero.   Various strategies are used for dealing with this issue in the semigroup setting: one approach is to work with certain weighted function spaces involving powers of $t$ which precisely capture regularity of the solutions down to $t = 0$. A second approach is to use continuous interpolation spaces between the Banach spaces $\md$ and $\mx$. For our work the first approach is adequate.  We thus introduce a series of function spaces that are used only in this section.  

For a time interval $(a,b)$, a H\"older exponent $\alpha \in (0,1)$, and a Banach space $\mx$, we introduce a semi-norm on curves from $(a,b)$ into $\mx$ by
\[ [ f ]_{C^{\alpha}( (a,b); \mx )} = \sup_{
s,t \in (a,b), s \neq t} \frac{\| f(t) - f(s) \|_\mx}{{|t-s|^{\alpha}}}, \]
and a corresponding norm by
\[ \|f\|_{C^{\alpha}( (a,b); \mx )} = \sup_{t \in (a,b)} \|f(t)\|_\mx + [ f ]_{C^{\alpha}( (a,b); \mx )}. \]
Denote the space of continuous curves for which this norm is finite by $C^{\alpha}( (a,b); \mx )$.  The (time) weighted space $C^{\alpha}_{\alpha}( (a,b); \mx)$ is given by curves for which the norm
\[ \|f\|_{C^{\alpha}_{\alpha}( (a,b); \mx )} = \sup_{t \in (a,b)} \|f(t)\|_\mx + [ (t-a)^{\alpha} f ]_{C^{\alpha}( (a,b); \mx )} \]
is finite.

We state an existence theorem in a form suitable for our work here. The estimate  \eqref{well-posedness local estimate} below is the key to the well-posedness result in Theorem \ref{well-posedness}. 
We denote an initial condition by $u_i$, and the solution that starts at $u_i$ by $u_i(t)$. 

\begin{theorem}[Amalgamation of Theorems 8.1.1 and Corollary 8.1.2 in \cite{Lun95}] 
\label{existence}

For $\md \subset \mx$  a continuous embedding of Banach spaces, and for  $\sO$ an open  subset of $\md$, let 
\[ F: \sO \longrightarrow \mx \] satisfy  hypotheses $H0$, $H1$, and $H2$ above.  For every  $u_0 \in \sO$ such that $F(u_0) \in \overline{\md}$, there exists $\delta = \delta(u_0)>0$ and $r = r(u_0) > 0$ such that for any $t_0 \in [0, r)$ there is a unique solution $u$ to  the (nonlinear) initial value problem
\begin{equation}
\label{F-IVP}
u'(t) = F(u(t)), \; \; u(t_0) = u_0. 
\end{equation}
The solution $u$ satisfies the regularity condition
\[ u \in C( [t_0,t_0+\delta]; \md ) \cap C^{1}( [t_0,t_0+\delta]; \mx ), \]
and moreover $u \in C^{\alpha}_{\alpha}( (t_0, t_0+\delta]; \md)$ and $u$ is the unique solution to the IVP in $$\bigcup_{\beta \in (0,1)}~C^{\beta}_{\beta}( (t_0, t_0+\delta]; \md) \cap C( [t_0,t_0+\delta]; \md ).$$
Additionally there is a constant $k = k(u_0) > 0$ such that if  $u_1\in \sO$ satisfies the inequality 
\[ \| u_1 -u_0\| \leq r, \; \; F(u_1) \in \overline{\md}, \]
then the solution $u_1(t)$ satisfies the inequality 
\begin{equation}
\label{well-posedness local pre-estimate}
 \|u_1(t) - u_0(t)\|_{C^{\alpha}_{\alpha}((t_0,t_0+\delta],\md)} \leq k \| u_1- u_0\|_\md. 
 \end{equation}
\end{theorem}

This theorem guarantees the short-time existence, uniqueness, and continuous dependence on initial conditions - in short {\it well-posedness}  - of the nonlinear initial value problem \eqref{F-IVP}.  In view of the definition of the norm for the space $C^{\alpha}_{\alpha}((t_0,t_0+\delta],\md)$, we may translate the estimate from equation \eqref{well-posedness local pre-estimate} to the more useful form
\begin{equation}
\label{well-posedness local estimate}
 \|u_1(t) - u_0(t)\|_{\md} \leq k \| u_1 - u_0\|_\md, \; \; \mbox{for all} \; t \in (t_0,t_0+\delta).
\end{equation}

\section{The Ricci-DeTurck flow}
\label{Sec:RDF}

As is well-known, the Ricci flow PDE system is not itself parabolic; to prove local existence (and well-posedness)  for Ricci flow, one uses the Ricci-DeTurck flow as an intermediary.  In this section, we review the setup of the Ricci-DeTurck flow system, with an emphasis on the analysis which leads to a proof that the Ricci-DeTurck initial value problem is well-posed (including continuous dependence on initial data) for data in the Banach spaces of interest for this work.

Let $\mfld$ be the manifold on which the evolving metrics $g(t)$ are defined, and let $\gtil$ be a fixed reference metric on $\mfld$.   Denoting the connection coefficients, the covariant derivative operators, and the curvatures corresponding to the reference metric $\gtil$ with a tilde, and those corresponding to the evolving metric  $g$ without a tilde, we define the DeTurck vector field 
\[ W^k = W^k(g,\gtil) := g^{pq} \left( \Gamma_{pq}^k - \widetilde{\Gamma}_{pq}^k \right), \] 
and then write the initial value problem for the Ricci-DeTurck flow  as follows:
\[ \partial_t g = F^{\gtil}(g), \; \; g(0) = g_0; \]
here $F^{\gtil}(g) := - 2 \Rc (g) - 2 \mathcal{L}_W g$, and the superscript on $F$ emphasizes the choice of background metric appearing through the vector field $W$.

If we study the Ricci-DeTurck flow for metrics $g$ close to the reference metric $\gtil$, then it is useful to consider the linearization of this flow relative to $\gtil$; we obtain (for $h := g-\gtil$) the linear evolution system 
\[\partial_t h= A^{\gtil}_{\gtil} h,\]
where the operator $A^{\gtil}_{\gtil}$ is the Lichnerowicz Laplacian (relative to $\gtil$): 
\[ A^{\gtil}_{\gtil} h := D_h F^{\gtil}(\gtil)  = \left. \frac{d}{ds} \right|_{s=0} F^{\gtil}( \gtil + sh ) = \tilde \Delta_L h. \]
One readily verifies the following explicit formula for  $A^{\gtil}_{\gtil} =  \tilde \Delta_L$:  
\[\tilde  \Delta_L h_{ij} := {\gtil}^{ab} \nabtil_a \nabtil_b h_{ij} + 2 \Rtil_{iabj} h^{ab} - \Rctil_{ia} h^{a}_j - \Rctil_{aj} h^{a}_i. \]

We also consider solutions that start near a fixed metric $\gee$ which is distinct from the reference metric $\gtil$. 
For this, we linearize the operator $F^{\gtil}$ 
about the metric $\gee$. Labeling this linearization operator $A^{\gtil}_{\gee}$, 
we calculate  (following Proposition 3.2 of \cite{GIK}) 
\begin{equation*}
A^{\gtil}_{\gee} h:= D_h F^{\gtil}(\gee) = \left. \frac{d}{ds} \right|_{s=0} F^{\gtil}( \gee + sh ) = \Delta_L^{\gee} h - \Psi_{\gtil} h, 
\end{equation*}
where $\Psi_{\gtil} h$ is a first-order linear operator in $h$ whose coefficients involve $\gee$-contractions of at most two $\gee$-covariant derivatives of $\gtil$.  This can be written in the following form 
\begin{equation} \label{form-of-A}
  A^{\gtil}_{\gee} h = a(\gee) \partial^2 h + b(\gee, \partial \gee, \gtil, \partial \gtil) \partial h + c(\gee, \partial \gee, \partial^2 \gee, \gtil, \partial \gtil, \partial^2 \gtil) h 
\end{equation}
for suitable components $a, b, c$.   A calculation shows that $A^{\gtil}_{\gee}$ is elliptic.

Having identified the operators of interest, we now discuss the spaces of tensors on which we apply Theorem \ref{existence}.  Let $\sym$ denote the bundle of symmetric $2$-tensor fields on $\mfld$.  Since $A^{\gtil}_{\gee}$ is a second-order elliptic operator on $\sym$, the H\"older spaces $C^{2,\alpha}(\sym)$ are a natural choice.  Unfortunately, $C^{2,\alpha}(\sym)$ is not a closed subset of $C^{0,\alpha}(\sym)$ in the $\alpha$-norm, and this makes the hypothesis of Theorem \ref{existence} involving the closure of the domain $\md$ difficult to check for generic initial data.  Instead we use ``little H\"older spaces'', since the inclusion $h^{2,\alpha} \hookrightarrow h^{0,\alpha}$ is continuous and dense (see Appendix \ref{Geometric background} for a detailed introduction of little H\"older spaces).  In particular let  $$P := A_{\gee}^{\gtil},$$ and set
\[ \md(P) = h^{2,\alpha}(\sym), \]
with 
\[ \mx = h^{0,\alpha}(\sym). \]
Observe that the closure of $\md(P)$ in $\mx$ is $\mx$: 
\[ \overline{\md(P)}= \mx. \]
In particular for any $g \in \md(P)$, $F(g) \in \mx = \overline{\md(P)}$.  We could also take the pair $(\md,\mx) = (h^{k+2,\alpha}(\sym), h^{k,\alpha}(\sym))$, but to keep the notation simple we restrict to $k=0$.  We now build toward verifying the hypotheses of Theorem \ref{existence} for a certain open subset of $\md(P)$.

\subsection{Sectoriality of the DeTurck operator} 

We must show that $P$ is sectorial between these function spaces.  Since this point has been treated tersely (and at times incorrectly) in the existing literature, we start at the beginning.  Note that the arguments we give in this section use the compactness of $\mfld$ in an essential way.

Assume that the chosen reference metric $\gtil$ and the metric $\gee$ at which we are linearizing, are smooth.  Then $P = A_{\gee}^{\gtil}$ is an unbounded self-adjoint elliptic operator with smooth coefficients on $L^2(\sym)$, with maximal domain $\md_{L^2}(P) := \{ w \in L^2(\sym): P w \in L^2(\sym) \}$. We know that there exists a finite $K \in \bR$ for which $\sigma_{L^2}(P) \subset (-\infty, K]$ (see e.g. Appendix Theorem \ref{thm:L2-spectrum}).  Thus the resolvent set $\rho_{L^2}(P)$ contains a sector in the complex plane.

We now analyze the spectrum of $P$  acting as an operator on  H\"older spaces $C^{k,\alpha}(\sym)$ for  $k \in \bN$ and $\alpha \in (0,1)$.  By analogy with the $L^2$ theory, we  consider the maximal domain for $P$:
\[ \md_{C^{0,\alpha}}(P) := \{ u \in C^{0,\alpha}(\sym) | Pu \in C^{0,\alpha}(\sym)\}. \]
By definition, the $C^{0,\alpha}$-resolvent set for $P$ is given by
\[ \rho_{C^{0,\alpha}}(P) = \left\{ \lambda \in \bC| \lambda I - P: \md(P) \to C^{0,\alpha}(\sym) \; \mbox{has bounded inverse} \right\}, \]
and the corresponding spectrum is
\[\sigma_{C^{0,\alpha}}(P) = \bC \smallsetminus \rho_{C^{0,\alpha}}(P).\]
 For $\lambda \in \rho_{C^{0,\alpha}}(P)$ we have  the resolvent operator,
 \begin{equation}
 \label{ResOp}
 R_{\lambda} := (\lambda I - P)^{-1}: C^{0,\alpha}(\sym) \to \md(P) \subset C^{0,\alpha}(\sym). 
 \end{equation}

We next give a proof of the following Proposition, which guarantees that the spectrum of $P$ on $C^{0,\alpha}(\sym)$ coincides with its $L^2$ spectrum.  This result relies on  elliptic regularity and on the Sobolev embedding theorem, as well as the compactness of $\mfld$.
\begin{prop} \label{spectra-coincide} Let ($\mfld,g)$ be a closed Riemannian manifold, and $P = A_{\gee}^{\gtil}$. 
For any  $\alpha \in (0,1)$, the $L^2$ and $C^{0,\alpha}$ spectra coincide:
\[ \sigma_{L^2(\sym)}( P ) = \sigma_{C^{0,\alpha}(\sym)}( P ). \]
\end{prop}
\begin{proof} 
We prove that these subsets of the complex plane coincide.  We first choose any $\lambda \in \sigma_{L^2}(P)$.  Since the $L^2$ spectrum of $P$ on a compact manifold is purely point spectrum, there is an $L^2$-eigentensor  field $u$ corresponding to $\lambda$ for which 
\[ (\lambda I - P) u = 0. \]
Elliptic regularity implies that $u$ is smooth, and thus that $u$ is an element of $\md_{C^{0,\alpha}}(P)$.  It follows  that $\lambda I - P$ as an operator on $C^{0,\alpha}$ has a non-trivial kernel; consequently  $\lambda \in \sigma_{C^{0,\alpha}}(P)$.  This tells us that $\lambda$ is  an eigenvalue for $P$ acting on $C^{0,\alpha}$ as well as on $L^2$.  

We now choose $\lambda \in \rho_{L^2}( P )$.  We know that $(\lambda I - P): L^2(\sym) \to L^2(\sym)$ is invertible with bounded inverse on its domain; i.e., there exists a constant $C_1 > 0$ such that  for all  $v \in L^2(\sym)$, one has 
\begin{equation} \label{l2-est}
 \| (\lambda I - P)^{-1} v  \|_{L^2} \leq C_1 \|v\|_{L^2}. 
\end{equation}
Since $C^{0,\alpha}(\sym) \subset L^2(\sym)$ on a compact manifold, the maximal H\"older domain $\md_{C^{\alpha}}(P)$ is defined and is contained in the maximal $L^2$ domain, $\md_{L^2}(P)$.  

To show that if $\lambda \in \rho_{L^2}( P )$ then $\lambda \in \rho_{C^{0,\alpha}(\sym)}(P)$ (thereby completing the proof of this proposition), we need to verify that  the estimate of equation \eqref{l2-est} holds in the H\"older norms, as well as in the $L^2$ norms.  Indeed, we first note that for every $ v \in C^{0,\alpha}(\sym)$, $u := (\lambda I - P)^{-1} v$ is a well-defined element of $L^2(\sym)$, for which the following  elliptic PDE holds
\begin{equation}
\label{vu}
v = (\lambda I - P) u. 
\end{equation}
We now begin a bootstrap procedure.  Since $C^{0,\alpha}(\sym) \subset L^p(\sym)$ for all $p$, we may regard $v \in L^p(\sym)$.  Hence, elliptic regularity guarantees that $u \in L^{2,p}(\sym)$.  Since $L^{2,p}(\sym) \subset L^p(\sym)$ is a compact embedding, and since $L^p(\sym) \subset L^2(\sym)$ is a continuous inclusion for $p \geq 2$, the elliptic estimate corresponding to \eqref{vu} yields
\[ \|u \|_{L^{2,p}} \leq C ( \|v\|_{L^p} + \|u\|_{L^2} ); \]
here we note that this constant $C$ depends on $\lambda$, since the operator appearing in \eqref{vu} involves $\lambda$.  Now, for $p$ sufficiently large (i.e., for $p$ large enough so that  $2- \frac{n}{p} > \alpha)$), Sobolev embedding tells us that $L^{2,p}(\sym) \subset C^{0,\alpha}(\sym)$.  Consequently, we find that 
\[ \|u\|_{C^{0,\alpha}} \leq \|u \|_{L^{2,p}} \leq C ( \|v\|_{L^p} + \|u\|_{L^2} )\leq C' \|v\|_{C^{0,\alpha}}. \] 
We thus verify that  $\| (\lambda I - P)^{-1} v  \|_{C^{0,\alpha}} \leq C \|v\|_{C^{0,\alpha}}.$  This shows that $\lambda \in \rho_{C^{0,\alpha}(\sym)}(P)$, which completes the proof.
\end{proof}

Proposition \ref{spectra-coincide} combined with the fact that the $L^2$-spectrum of $P$ is contained in the ray $(-\infty, K]$ show that $\rho_{C^{0,\alpha}}(P)$ contains a sector. 

The second part of the proof of the sectoriality of a given operator is the verification of the resolvent estimate \eqref{ResEst}.
This requires that we obtain uniform control of the constant $M$, and thus requires some care.  In the literature it has been stated that the resolvent estimate follows from standard Schauder theory; however,  since the operator $\lambda I  - P$ depends on $\lambda$, use of the Schauder estimates results in constants which depend on $\lambda$. Instead, we adapt a scaling argument from \cite{BM} to obtain the resolvent estimate.  The reader may compare our discussion with that in  Chapter 3 of \cite{Lun95}; there, resolvent  estimates (on $\bR^n$) are obtained  from the classical Agmon-Douglis-Nirenberg estimates for operators with less regular coefficients.

\begin{prop} \label{resolvent-estimate} 
Let $(\mfld,g)$ be a compact smooth Riemannian manifold, and let $ \Delta_L$ denote the Lichnerowicz Laplacian with respect to $g$ acting on symmetric $2$-tensors. For any positive constant $K$ greater than the maximum eigenvalue of  $ \Delta_L$  and for a fixed  $\alpha \in (0,1)$, there exists $M > 0$ such that for all $\lambda$ with $Re(\lambda) > K$, and for all $f \in C^{0,\alpha}(\sym)$, the resolvent operator $R_\lambda$ (see \eqref{ResOp}) satisfies the inequality
\[ \| R_{\lambda} f \|_{\mathcal{L}(C^{0,\alpha})} \leq \frac{M}{|\lambda|} . \]
\end{prop}

\begin{proof} As a preliminary step, by rescaling the metric we may assume that the injectivity radius is bounded below by $1$.

We proceed by contradiction:  If the estimate is false, then for every $n \in \bN$ it must be possible to find a number $\lambda_n$ with $Re(\lambda_n) > K$ and a tensor field  $f_n \in C^{0,\alpha}(\sym)$ for which the following holds:
\begin{equation} \label{eqn:bwoc-holder} \| R_{\lambda_n} f_n \|_{C^{0,\alpha}} > \frac{n}{|\lambda_n|} \| f_n \|_{C^{0,\alpha}}. \end{equation} 

If we further assume that the sequence of numbers $\lambda_n$ remains in a compact set, then some subsequence must converge to some $\lambda_*$ with $ Re(\lambda_*) \geq K$.  But then along this subsequence, the estimate of equation \eqref{eqn:bwoc-holder} above shows that $\|R_{\lambda_*}\|_{\mathcal{L}(C^{0,\alpha})}$ must diverge, which is impossible since $\lambda_*$ is not contained in the H\"older spectrum of $\Delta_L$.  Consequently,  we may now assume that $\lambda_n \to \infty$ (note that $\lambda_n \in \bC$, so this  means that $|\lambda_n| \to +\infty$).

We consider the sequence of tensor fields
\[ w_n(x) := \frac{R_{\lambda_n} f_n(x)}{\| R_{\lambda_n} f_n \|_{C^{0,\alpha}}}. \]
Applying the operator $I - \lambda_n^{-1} \Delta_L$ to each $w_n$, and recalling the definition of the resolvent operator  $R_{\lambda_n}$ (see \eqref{ResOp}), we obtain
\begin{equation} \label{eqn:limit} 
 (I - \lambda_n^{-1} \Delta_L) w_n = \lambda_n^{-1} ( \lambda_n I - \Delta_L )w_n = \lambda_n^{-1} \frac{f_n(x)}{\|R_{\lambda_n} f_n\|_{C^{0,\alpha}}}.
\end{equation}
It is an immediate consequence of inequality  \eqref{eqn:bwoc-holder} that the $C^{0,\alpha}$ norm of the right hand side of equation \eqref{eqn:limit} converges to zero as $n \to \infty$. Pulling the analysis back to normal coordinates centred at certain points, we now show that this limiting behavior leads to a contradiction. 

It follows from the compactness of the manifold $\mfld$ and from the continuity of $R_{\lambda_n}f_n$ that for  each $n \in \bN$, there is a point $x_n$ at which  the supremum of $R_{\lambda_n} f_n$ is achieved;  i.e., we have \[ \| R_{\lambda_n} f_n \|_{L^{\infty}(\mfld)} = \sup_{x\in \mfld} |R_{\lambda_n} f_n(x)| = |R_{\lambda_n} f_n(x_n)| \]
(of course the norm of a $2$-tensor $\zeta_{ij}$ is given by $| \zeta |^2 = g^{ij}g^{kl}  \zeta_{ik} \zeta_{jl}$).

Also as a consequence of the compactness of $\mfld$, we know that there exists a subsequence of the sequence $\{x_n\}$ which converges to some point $x_\infty \in \mfld$. Relabeling this subsequence as $\{x_n\}$, and letting $B_n$ denote the unit ball with centre $x_n$, we now introduce geodesic normal coordinates $z$ in each ball, centred at $x_n$ (to simplify the notation, we suppress any labelling of these geodesic coordinates related to the index $``n"$).  In terms of these coordinates (in each ball $B_n$), the metric components takes the special form
 $g_{ij} = \delta_{ij} + O(|z|^2)$, and moreover since compact smooth Riemannian manifolds are of bounded geometry, any number of derivatives of the metric are uniformly bounded in normal coordinates \cite[Proposition 1.2]{Eichhorn}.  Finally the tensor bundle $\sym$ is trivialized in terms of these coordinates over $B_n$.  Now return to equation \eqref{eqn:limit} and for each $n$ express in these normal coordinates.  The components then read
\begin{equation}
\label{eqn:components1}
 \lambda_n (w_n)_{ij} - (\Delta_L w_n)_{ij}  = \frac{(f_n(z))_{ij}}{\|R_{\lambda_n} f_n\|_{C^{0,\alpha}}}.
\end{equation}
 
In terms of geodesic coordinates (in each ball), the Lichnerowicz Laplacian---acting on a symmetric tensor field $u$--- takes the following form
 \begin{align*}
 \label{DeltaLz}
  (\Delta_L u)_{ij} &= {g}^{kl} \nabla_k \nabla_l u_{ij} + 2 \Rm_{iklj} u^{kl} - \Rc_{ik} u^{k}_j - \Rc_{kj} u^{k}_i\\
 &= \delta^{kl} \frac{\partial^2 u_{ij}}{\partial z^k \partial z^l}  + (S(u) + 2\Rm*u - 2\Rc*u)_{ij},
 \end{align*}
where  $S$ is a second-order operator in $u$ with bounded smooth coefficients that vanishes to order $|z|^2$ and $*$ indicates  contractions with respect to the metric $g$.  Using this we find equation \eqref{eqn:components1} now reads
\begin{equation}
\label{eqn:components2}
 \lambda_n (w_n)_{ij} - \delta^{kl} \frac{\partial^2 (w_n)_{ij}}{\partial z^k \partial z^l}  - (S(w_n) + 2\Rm*w_n - 2\Rc*w_n)_{ij}  = \frac{(f_n(z))_{ij}}{\|R_{\lambda_n} f_n\|_{C^{0,\alpha}}},
\end{equation}
which we regard as a sequence of a system of equations on $B_1(0)$.  Introduce the operator $Z$ whose action on $w_n$ is given by
\[ Z w_n := S(w_n) +  2\Rm*w_n - 2\Rc*w_n. \] 

Now rescale each system by introducing coordinates $y := \sqrt{|\lambda_n|} z$. In doing so, we  observe that since  the coordinates $z$ are valid on $B_1(0)$, it follows that  the coordinates $y$ are valid on $B_{|\lambda_n|}(0)$.  Computing $S$ from \eqref{eqn:components2} in these rescaled coordinates using the chain rule $\partial_z = \sqrt{|\lambda_n|} \partial_y$ and $\partial^2_z = {|\lambda_n|} \partial^2_y$ and the asserted facts about the metric in normal coordinates allows us to write
\[ Z w_n = |\lambda_n|\widetilde{Z} w_n, \]
where $\widetilde{Z}$ has coefficients that converge to zero on compact subsets.  Thus we rewrite equation \eqref{eqn:components2} in the rescaled $y$ coordinates as
\begin{equation}
\label{eqn:components3}
 \lambda_n (w_n)_{ij} - |\lambda_n| \delta^{kl} \frac{\partial^2 (w_n)_{ij}}{\partial y^k \partial y^l}  - |\lambda_n| \widetilde{Z}(w_n)_{ij}
   = \frac{\left(f_n\left( \frac{1}{\sqrt{|\lambda_n|}} y \right)\right)_{ij}}{\|R_{\lambda_n} f_n\|_{C^{0,\alpha}}}.
\end{equation}
Set
\[ \mathcal{E}_1 = \frac{f_n\left( \frac{1}{\sqrt{|\lambda_n|}} y \right)}{ |\lambda_n| \cdot \|R_{\lambda_n} f_n\|_{C^{0,\alpha}}}, \]
from here we can write equation \eqref{eqn:components3} as
\begin{equation}
\label{eqn:components4}
 \frac{\lambda_n}{|\lambda_n|} (w_n)_{ij} -  \delta^{kl} \frac{\partial^2 (w_n)_{ij}}{\partial y^k \partial y^l} + \widetilde{Z}(w_n)_{ij} = [\mathcal{E}_1]_{ij}.
\end{equation}

Now consider the behaviour of the right hand side as $n \to \infty$.  By the remarks following equation \eqref{eqn:limit}, we have $\| \mathcal{E}_1 \|_{L^{\infty}(B_{|\lambda_n|}(0))} \leq \| \mathcal{E}_1 \|_{C^{0,\alpha}(B_{|\lambda_n|}(0))} \to 0$.  

To summarize the argument so far, from equation \eqref{eqn:components4} we have a sequence of equations on $\bR^n$ that are eventually defined on a ball of any radius, and whose right hand side converges to zero uniformly on compact sets.  Thus on $B_1(0)$ we may apply local elliptic regularity to write
\begin{equation} \label{eqn:loc-elliptic-reg} 
\| w_n \|_{C^{2,\alpha}(B_1(0))} \leq c \left( \|\mathcal{E}_1 \|_{ {C^{0,\alpha}(B_2(0)) }} + \|w_n\|_{C^{0,\alpha}(B_2(0))} \right), 
\end{equation}
where $c$ is uniformly bounded in $n$ since the coefficient $\lambda_n /|\lambda_n|$ is of unit modulus.  By compactness of the unit sphere and the Arzela-Ascoli theorem we may obtain a subsequence $\lambda_{n_1}$ of $\lambda_n$ and a subsequence $w_{n_1}$ of $w_n$ so that both $\lambda_{{n_1}_k} / |\lambda_{{n_1}_k}|$ converges to a unit modulus element $e \in \bC$, and $w_{n_1}$ converges in $C^{2,0}(B_1(0))$.  

Now return to equation \eqref{eqn:loc-elliptic-reg}, modify the radii involved and repeat this procedure to obtain a subsequence $n_2$ of $n_1$ so that $w_{n_1}$ converges in $C^{2,0}(B_2(0))$.  Proceeding inductively we obtain a subsequences $n_k$ of $n_{k-1}$ so that $w_{n_k}$ converges in $C^{2,0}(B_k(0))$.  Finally, take the diagonal subsequence $w_{{k_k}}$ and set $w_* = \lim_k w_{{k_k}}$, which we may regards as a $C^2$ vector-valued function $w_*$ defined on flat $\bR^n$ that satisfies 
\[ (e-\Delta) w_* = 0, \]
where here $\Delta$ acts as the flat Laplacian on each component.
Moreover, $ \|w_*\|_{L^{\infty}} \leq 1$ and $|w_*(0)| = 1$, so $w_* \neq 0$.  However note that $e\neq -1$ since $Re(\lambda_n) > K$.  A short computation using Fourier analysis shows that no solution can exist, since the total symbol of $e - \Delta$ is nonvanishing.   This contradiction concludes the proof.
\end{proof}

We are now ready to prove that  $\Delta_L$ is a sectorial operator in H\"older spaces. Recall that $P = A_{\gee}^{\gtil}$ is in fact $\Delta_L$ when $\gee=\gtil.$  
\begin{theorem}
Let $(\mfld,g)$ be a closed Riemannian manifold. Then for any $\alpha \in  (0,1)$, 
\[ \Delta_L: \md \subset C^{0,\alpha}(\sym) \to C^{0,\alpha}(\sym) \]
is a sectorial operator. 
\end{theorem}
\begin{proof}
Since $\Delta_L$ is an unbounded operator on $L^2$, then Theorem \ref{thm:L2-spectrum} implies there exists $K \in \bR$ so that $\sigma_{L^2}(\Delta_L) \subset (-\infty, K]$.  Thus the $L^2$ resolvent set contains a sector, and Proposition \ref{spectra-coincide} allows us to conclude that the $C^{0,\alpha}$ resolvent set also contains a sector.  Proposition \ref{resolvent-estimate} gives the resolvent estimate in a half-plane.  However knowing the resolvent set contains a half-plane and a uniform estimate in this half-plane is a sufficient condition for sectoriality by Proposition 2.1.11 of \cite{Lun95}.
\end{proof}

\begin{corollary}
If $(\mfld,g)$ is a closed Riemannian manifold. Then for any $\alpha \in  (0,1)$, 
\[ \Delta_L: \md \subset h^{0,\alpha}(\sym) \to h^{0,\alpha}(\sym) \]
is a sectorial operator.
\end{corollary}
\begin{proof}  This simply follows by continuity of the resolvent and the fact that smooth tensor fields are dense in $h^{0,\alpha}$. \end{proof}

\subsection{Continuous Dependence of the Ricci-DeTurck flow}

We are in a position to prove continuous dependence of the Ricci-DeTurck flow.  Choose any smooth initial metric $g_0$ and set the reference metric $\gtil = g_0$.  Recall that in this exposition we take $\md = h^{2,\alpha}(\sym)$ and $\mx = h^{0,\alpha}(\sym)$, but we may also use $\md = h^{k+2,\alpha}(\sym)$ and $\mx = h^{k,\alpha}(\sym)$.  The Ricci-DeTurck operator is a nonlinear map $F^{g_0}: \md \longrightarrow \mx$, and is Fr\'echet differentiable at any point since it is a polynomial contraction of up to two covariant derivatives of its argument.  Consider the open ball in $\md$ defined by
\[ \sO := B_r( g_0 ) = \{ g \in \md: \|g - g_0\|_{h^{2,\alpha}(\sym)} < r \}, \]
where we choose $r > 0$ sufficiently small that the following three conditions are satisfied (see also \cite{KnopfYoung}):
\begin{enumerate}
\item Every $g$ in $B_r(g_0)$ is a Riemannian metric.
\item For each $g$ in $B_r(g_0)$, the linearization $A^{g_0}_g$ is uniformly elliptic.  
\item For all $g \in \mo$, $h \in h^{2,\alpha}(\sym)$, 
\[ \| (A^{g_0}_g - A^{g_0}_{g_0}) h \|_{h^{0,\alpha}(\sym)} < \frac{1}{M+1} \|h\|_{h^{2,\alpha}(\sym)}, \] 
\end{enumerate}
where $M$ is the constant from Proposition \ref{resolvent-estimate}.

We remark that for item (2), the considerations that lead to equation \eqref{form-of-A} show that the principal symbol of $A^{g_0}_g$ involves only the inverse of the metric, $g^{-1}$. See \cite{KnopfYoung, GIK} for further details. For item (3), standard estimation (see Section 4 of \cite{KnopfYoung}) allows one to conclude that such a bound exists.  

We next verify hypothesis H1 of Theorem \ref{existence}.  Our argument so far only proves the sectoriality of $A^{g_0}_{g_0}$, but we need sectoriality of $A^{g_0}_g$ in an open set.  In view of the third item above, for any $g \in \sO$, $A^{g_0}_g = (A^{g_0}_g - A^{g_0}_{g_0}) + A^{g_0}_{g_0}$ is a small perturbation of a sectorial operator on $\mx$ and hence is sectorial on $\mx$ by \cite[Proposition 2.4.2]{Lun95}.  Moreover, by elliptic regularity the graph norm of $A^{g_0}_{g}$ is equivalent to the $h^{2,\alpha}$-norm.  This verifies hypothesis H1.

The hypothesis H2 follows from the structure of $A$ as a polynomial contraction of the components of equation \eqref{form-of-A} and may be verified in a manner similar to Section 4 of \cite{KnopfYoung}.
We have satisfied the hypotheses of Theorem \ref{existence}, and thus have existence, uniqueness, and continuous dependence on the initial conditions for a short time.  

We next study continuous dependence on initial conditions for the entire interval of existence. In what follows, $g_i$ means the fixed metric at $t=0$, and $g_i(t)$ is the flow that starts at $g_i$.

Let $\tau(g_0)$ be the \textit{maximal time of existence} of a solution that starts at $g_0$.   We conclude this section with the proof of continuous dependence of the Ricci-DeTurck flow.  We use the pair $(\md,\mx) = (h^{k+2,\alpha}(\sym), h^{k,\alpha}(\sym))$ in this theorem since our eventual application next section requires $k=2$.

\begin{theorem}[Continuous Dependence of the Ricci-DeTurck flow]
\label{well-posedness}
Let $(\mfld,g_0)$ be a compact Riemannian manifold, $\md = h^{k+2,\alpha}(\sym)$, $k \ge 0$, and let ${g_0}(t), \  {g_0}(0) = g_0$
be a solution of the Ricci-DeTurck flow with background metric $g_0$,  so the flow exists on a maximal time interval $[0,\tau(g_0)). $ Let $\tau < \tau(g_0).$ Then there exist constants $r > 0$ and $C>0$ depending on $g_0$ and $\tau$, such that if $$ ||g_1 - g_0||_{h^{k+2,\alpha}(\sym)} \le r,$$ then
\[\tau(g_1) \ge \tau,\]
and
\begin{equation*}
||{g}_1(t) - g_0(t) ||_{h^{k+2,\alpha}(\sym)}   \le C || g_1 - g_0 ||_{h^{k+2,\alpha}(\sym)}
\end{equation*} 
for all $t \in [0,\tau].$
\end{theorem}
\begin{proof}
The argument we give is an adaptation of the covering argument of \cite[Proposition 8.2.3]{Lun95}.

Given $g_0(s)$ for $s \in [0,\tau]$ as in the statement, observe from Theorem \ref{existence} that for each $s\in[0,\tau]$ there exists positive constants $r_s, \delta_s$ and $c_s$ so that if $|t_0 - s| \leq r_s$ and $g_1 \in D$ with $\| g_1 - g_0(s) \|_{\md} \leq r_s$ and $F(g_1) \in \overline{\md}$, then the flow has a unique solution on $[s,s+\delta_s]$, with
\begin{equation*}
||{g}_1(s) - g_0(s) ||_{h^{k+2,\alpha}(\sym)}   \le c_s || g_1 - g_0 ||_{h^{k+2,\alpha}(\sym)},
\end{equation*}  
on that short interval.

As the set 
\begin{equation*}
\label{cpt set}
[0,\tau] \times \{g_0(t): t \in [0,\tau]\} 
\end{equation*}
is compact in $\mathbb R \times \md$, it is covered by finitely many open subsets $\{ U_i: i=1, \cdots, N\}$ where
\[ U_i = (s_i - r_{s_i}, s_i + r_{s_i}) \times B_{r_{s_i}}^{\md}(g_0(s_i)). \]
Obtain positive constants $\delta = \min_{1 \leq i \leq N} \delta_{s_i}$, $c = \max_{1 \leq i \leq N} c_{s_i}$, $r = \min_{1 \leq i \leq N} r_{s_i}$. 

Now $(0,g_0) \in U_{i_0}$ for some $i_0$.  Thus for $\|g_1 - g_0\|_{\md} \leq r$, 
\[ ||g_1(t) - g_0(t)||_\md \le c ||g_1 - g_0||_\md, \; \; \mbox{for}  \; t \in  [0,\delta]. \]
If $\delta > \tau$ the proof is complete, otherwise observe $(\frac{1}{2}\delta,g_0(\frac{1}{2}\delta)) \in U_{i_1}$ for some $i_1$, and we may solve and obtain estimates on a further interval of time of length $\delta$ to obtain $\tau(g_0) > \frac{3}{2} \delta$ and
\[ ||g_1(t) - g_0(t)||_\md \le c_1 ||g_1 - g_0||_\md, \; \; \mbox{for}  \; t \in  \left[0,\frac{3}{2}\delta\right], \]
for a new constant $c_1$.  This process terminates after finitely many steps.
\end{proof}

\section{Continuous Dependence of the Ricci flow}
\label{Sec:ContDepRF}

In this section we prove continuous dependence of the Ricci flow on initial conditions. We assume closeness of initial conditions in a slightly more regular space, due to loss of regularity when moving from the Ricci-DeTurck flow to the Ricci flow.  The next theorem quantifies the idea that a sufficiently small perturbation (in $h^{4,\alpha}$) of a smooth metric yields a Ricci flow that remains nearby, as quantified by a $h^{2,\alpha}$ dependence estimate.  This perturbed Ricci flow is smooth after $t = 0$.

\

\noindent{\bf Theorem A} (Continuous Dependence of the Ricci Flow). \label{Continuous dependence of the Ricci flow} \emph{Let $(\mfld,g_0)$ be a compact Riemannian manifold, with $g_0$ a smooth metric, and let $g_0(t), \ \ g(0) = g_0$  be the maximal solution of the Ricci flow that exists for time $t \in [0,\tau(g_0))$, with $\tau(g_0) \le \infty$. Choose $\tau < \tau(g_0)$, $k \geq 2$. Then  there exist positive constants $r$ and $C$ depending only on $g_0$ and $\tau$ such that if 
\[ ||g_1 - g_0||_{h^{k+2,\alpha}(\sym)} \le r,\] then for $g_1(t)$ the unique solution of the Ricci flow starting at $g_1$,
\[ \tau(g_1) \ge \tau, \] and
\begin{equation*}
||g_1(t) - g_0(t) ||_{h^{k,\alpha}(\sym)}  \le C || g_1 - g_0 ||_{h^{k+2,\alpha}(\sym)} 
\end{equation*} 
for all $t \in [0,\tau].$}

\

\begin{proof}

We prove this in the case $k = 2$, but the proof is the same in the general case. We use both the Ricci flow and the Ricci-DeTurck flow.  Unadorned metrics depending on time, for example $g(t)$, refer to solutions of the Ricci flow, whereas time-dependent metrics with hats, for example $\hat{g}(t)$ are solutions to the associated Ricci-DeTurck flow.

Let $g_0(t), \  g(0) = g_0$ be the solution of the Ricci flow defined on $[0,\tau(g_0)),$ where $ \ \tau(g_0) \le \infty$. This generates a maximal solution to the Ricci-DeTurck flow $\hat{g}_0(t)$ with background metric equal to $g_0$, that starts at $g_0$ and exists at least until time $\tau(g_0)$. To see this, note that if $g(t)$ is a solution of the Ricci flow and $\phi(t)$ satisfies the harmonic map heat flow
\begin{equation*}
\ppt \phi_t = \Delta_{g(t),g_0} \phi_t,
\end{equation*}
then $\hat{g}(t) = (\phi_t)_* g(t)$ satisfies the Ricci-DeTurck flow \cite[pg. 121]{HRF}. Existence of the Ricci-DeTurck flow then follows from existence for the harmonic map heat flow.

Let $\tau < \tau(g_0)$.  By Theorem \ref{well-posedness}, there exists $\delta > 0$ so that if $g_1$ is a metric that satisfies
\[ \| g_1 - g_0 \|_{h^{4,\alpha}} < \delta, \]
then there is a unique Ricci-DeTurck flow $\hat{g}_1(t)$ starting at $g_1$ that exists at least until time $\tau$ and there exists a constant $L > 0$ such that
\begin{equation}
\label{RDF-dependence}
\| \hat{g}_1(t) - \hat{g}_0(t) \|_{h^{4,\alpha}} \leq L \| g_1 - g_0 \|_{h^{4,\alpha}}.
\end{equation}
Note that $\hat{g}_1(t)$ is smooth for any fixed time $t > 0$ by parabolic regularity.

Given $\delta$ above, fix $g_1$ with $\| g_1 - g_0 \|_{h^{4,\alpha}} < \delta$.  The DeTurck vector field with respect to the background metric $g_0$ is
\begin{equation}
\label{WDef}
W^k_i = W^k_i(t, \hat{g}_i(t), g_0) =\hat{g_i}(t)^{pq} ( \Gamma(\hat{g_i}(t))_{pq}^k - \Gamma(g_0)_{pq}^k ) 
\end{equation}
for $i = 0, 1$. If we solve the following ODE for the (time-dependent) diffeomorphisms $F_i(t)$ generated by $W_i$,
\begin{equation}
\label{diffeo}
 \partial_t F_i(t) = -W_i \circ F(t), \; \; \;  F_i(0) = Id, 
\end{equation}
where $Id$ is the identity map, we find that the metric $g_i(t) = F^*_i \hat{g}_i(t)$ is a solution of the Ricci flow starting at $g_i$, $i = 0, 1$. Note that at this point, we have lost control of one spatial derivative of our estimates due to \eqref{WDef}, as the $W^k_i$ are controlled in $h^{3,\alpha}$. 

\subsubsection{Estimates for the DeTurck vector field and diffeomorphism}

This subsection is rather technical, involving estimates for the $W$ and $F$.

Let us begin with the easier $W_0$.  We have supposed that the Ricci-DeTurck flow $\hat{g}_0(t)$ exists on $[0,\tau]$, so the curvature tensor is bounded on $[0,\tau]$, i.e. that there exists $K > 0$ where $|\Rm^{\hat{g}_0(t)}(x,t)|\leq K$ for all $x \in \mfld$ and $t \in [0,\tau]$.  Thus by Proposition 6.48 of \cite{ChowKnopf}, choosing $g_0$ as the reference metric, for any $m \in \bN$ there is a constant $C_m = C_m( \tau, K ) > 0$ where $| \nabla^m_{g_0} \hat{g}_0(x,t)|_{g_0} \leq C_m$ for all $x \in \mfld$ and $t \in [0,\tau]$.  In short: any number of derivatives of $\hat{g}_0$ are bounded by a constant depending on $\tau$ and $K$, i.e. there exists a constant $C(\tau, K, k, \alpha) > 0$ where
\begin{equation}
\label{bounds-on-g0}
\| \hat{g}_0(t) \|_{C^{k,\alpha}(\mfld)} \leq C(\tau, K, k, \alpha),
\end{equation}
for all $t \in [0,\tau]$.  This then implies that any norm of $W_0$ satisfies similar bounds on $[0,\tau]$.  

Regarding $W_1$, we must estimate in terms of the difference $\hat{g}_1(t) - \hat{g}_0(t)$ and $\hat{g}_0(t)$.  Omitting indices and time-dependence we must estimate
\begin{align*}
W_1 = \hat{g_1}^{-1} ( \Gamma(\hat{g_1}) - \Gamma(g_0) )
\end{align*}
by interpolating differences.  Since $W_1$ is a polynomial contraction in $\hat{g}_1^{-1}$ and $\partial \hat{g}_1$, one obtains abstractly
\[ W_1 = \mathcal{P}( \hat{g}_1 - \hat{g}_0, \partial(\hat{g}_1 - \hat{g}_0), \hat{g}_0, \partial \hat{g}_0 ), \]
where $\mathcal{P}$ denotes polynomial contractions in the tensors indicated.  In view of equation \eqref{RDF-dependence} and equation \eqref{bounds-on-g0}, we then see there is a constant depending on $\hat{g}_0$, $\tau$, $K$, $L$ and the algebraic structure of $\mathcal{P}$ so that
\[ \| W_1(t) \|_{h^{3,\alpha}} \leq C(\tau, K, L) \| {g}_1 - {g}_0 \|_{h^{4,\alpha}}, \]
for all $t \in [0,\tau]$, and similarly for the difference $W_1 - W_0$.  Taking a spatial derivative of $W_1$ denoted here by $\partial$, we obtain estimates of the form
\[ \| \partial W_1(t) \|_{h^{2,\alpha}} \leq C(\tau, K, L) \| {g}_1 - {g}_0 \|_{h^{4,\alpha}}, \]
and similarly for $W_1 - W_0$.

To obtain estimates for $F$, we may write equation \eqref{diffeo} abstractly as
\[ F_i(t) = -\int_0^t W_i( F(s) ) ds + Id. \] 
From here we can argue that any number of derivatives of $F_0(t)$ remain bounded by a constant that grows at worst linearly in $\tau$.  Up to three spatial derivatives of $F_1(t)$ remain bounded by a constant that grows linearly in $\tau$, and the same for the difference $F_1 - F_0$.  In summary, we have an estimate of the form
\begin{equation} \label{loc-est-f1}
 \| \partial F_1(t) \|_{h^{2,\alpha}} \leq C'(\tau, K, L) \| {g}_1 - {g}_0 \|_{h^{4,\alpha}}, 
\end{equation}
and
\begin{equation} \label{loc-est-f1-f0}
 \| \partial F_1(t) - \partial F_0(t) \|_{h^{2,\alpha}} \leq C''(\tau, K, L) \| {g}_1 - {g}_0 \|_{h^{4,\alpha}}.
\end{equation}
The precise dependence on $\tau$ does not matter for finite time (which is all that is required).

\subsubsection{Returning to the Ricci flow}

We are now ready to compare $g_1(t)$ and $g_0(t)$.  We work in coordinates, and we begin with a generic computation.  Suppose $\hat{g}$ is a metric expressed in local coordinates $\{(V,y^i)\}$ by
\[ \hat{g} = \hat{g}_{ij}(y) dy^i dy^j. \]
Let $F$ be a diffeomorphism $F: U \to V$, $y = F(x)$.  The pullback metric is
\[ g = F^* \hat{g} = \hat{g}_{ij}(F(x)) d (  y^i \circ F ) d (  y^j \circ F ) = \hat{g}_{ij}(F(x)) \pd{F^i}{x^a} \pd{F^j}{x^b} dx^a dx^b, \] 
which leads to the classical transformation law
\[ g_{ab}(x) = \hat{g}_{ij}(F(x)) \pd{F^i}{x^a}(x) \pd{F^j}{x^b}(x), \]
where $F^i = y^i \circ F$ is the $i$th component of $F$.
Now return to the Ricci flow argument  and let \[
F_i: (\mfld,\{x^{a}\},g_i(t)) \longrightarrow (\mfld,\{y^j\},\hat{g}_i(t)).\] We have, in a reference coordinate patch after suppressing time-dependence (note the inclusion of indices below is for reference)

\[ \| (g_1)_{ab} - (g_0)_{ab} \|_{h^{2,\alpha}} = \left\| (\hat{g}_1)_{ij}(F_1(x)) \pd{F_1^i}{x^a}(x) \pd{F_1^j}{x^b}(x) - (\hat{g}_0)_{ij}(F_0(x)) \pd{F_0^i}{x^a}(x) \pd{F_0^j}{x^b}(x) \right\|_{h^{2,\alpha}}, \]
so that the standard interpolation trick lets us write
\begin{align*} \| (g_1)_{ab} - (g_0)_{ab} \|_{h^{2,\alpha}} &\leq \left\| \left[(\hat{g}_1)_{ij}(F_1(x)) - (\hat{g}_0)_{ij}(F_0(x))\right] \pd{F_1^i}{x^a}(x) \pd{F_1^j}{x^b}(x) \right\|_{h^{2,\alpha}} \\ 
&+\left\|  (\hat{g}_0)_{ij}(F_0(x)) \left(\pd{F_0^i}{x^a}(x) - \pd{F_1^i}{x^a}(x)  \right) \pd{F_1^j}{x^b}(x) \right\|_{h^{2,\alpha}} \\
&+\left\|  (\hat{g}_0)_{ij}(F_0(x)) \pd{F_0^i}{x^a}(x) \left(\pd{F_0^j}{x^b}(x) - \pd{F_1^i}{x^b}(x)  \right) \right\|_{h^{2,\alpha}}.
\end{align*}
Taking this inequality, reinstating time-dependence, setting $r = \delta$ and then combining estimates \eqref{RDF-dependence}, \eqref{bounds-on-g0}, \eqref{loc-est-f1} and \eqref{loc-est-f1-f0} above we obtain a new constant,  $C > 0$ independent of $g_1$,   where for all $t \in [0,\tau]$,
\begin{align*} 
\| (g_1)_{ab}(t) - (g_0)_{ab}(t) \|_{h^{2,\alpha}} & \leq C \| g_1 - g_0 \|_{h^{4,\alpha}}. 
\end{align*}

Finally, we remark that by chasing the regularity of the metric throughout the proof above, one obtains a smooth dependence result with an estimate in terms of $h^{k,\alpha}$ norm of the initial data. 

\end{proof}

\section{Convergence Stability of the Ricci Flow and Applications}
\label{ConvStab}

In this section, we apply the continuous dependence result to prove convergence stability of the Ricci flow in several cases.

We begin with the case of a flat metric, $g_0$, on a torus $\bT^n$.  Since the operator has eigenvalues with zero real part, one expects a centre manifold. The stability of the Ricci flow on a torus is given in the following theorem (see Theorem 3.7 in \cite{GIK}):

\

\noindent{\bf Theorem} (Stability (Flat)\cite{GIK}).  \label{thm-GIK}
{\it Given a flat fixed point $g_0$ of the Ricci flow on a closed torus $\bT^n$, there exists a neighbourhood $\mathcal{U} \subset h^{2,\alpha}(\sym)$  of $g_0$ such that for any initial metric $g_1 \in \mathcal{U}$, the Ricci flow $g_1(t)$ converges in $h^{2,\alpha}(\sym)$ to a metric in the centre manifold of flat metrics at $g_0$.}

\

With this theorem and continuous dependence in hand, we obtain convergence stability (this is essentially Corollary 3.8 of \cite{GIK}, with the continuous dependence of the Ricci flow now proved).

\

\noindent{\bf Theorem B.} (Convergence Stability (Flat)).  \label{Convergence to Flat}
{\it Let $(\bT^n, g_0)$ be a smooth closed torus, where $g_0(t), \ g(0) = g_0$  is a solution of the Ricci flow that converges to a flat metric $g_\infty$. Then there exists $r$ such that if $g_1 \in B_r(g_0) \subset h^{2,\alpha}(\sym)$, then the solution $g_1(t), \ g(0) = g_1$ of the Ricci flow exists for all time and converges to a flat metric.}

\

\begin{proof}  This follows from the continuous dependence of the Ricci flow, together with stability of the fixed point $g_\infty$. 
Let $g_0(t), \  g(0) = g_0$ be a solution of the Ricci flow that converges to a flat metric $g_\infty$. 
Since $g_\infty$ is a flat fixed point of the Ricci flow, there exists a stability neighbourhood $\mathcal{U} \subset h^{2,\alpha}$ around $g_\infty$ by the stability theorem above. That is, the Ricci flow of any initial metric within this set converges to a flat metric in the centre manifold at $g_\infty$. 

Since $g_0(t)$ converges to $g_\infty$, we know that there exists finite $\tau$ such that $g_0(t) \in \mathcal{U}$ for all $\tau \le t \le \tau(g_0)$.  Since $\mathcal{U}$ is open, there exists $\rho > 0$ such that $B_\rho(g_0(\tau)) \subset \mathcal{U}$. By the continuous dependence of the Ricci flow,  we can choose $r > 0$  so that if $g_1$ is a metric that satisfies $||g_0 - g_1||_{h^{4,\alpha}} < r,$ then $||g_0(\tau) - g_1(\tau)||_{h^{2,\alpha}} \le C ||g_0 - g_1||_{h^{4,\alpha}}$.  This bound persists by shrinking the distance between $g_0$ and $g_1$; thus, if $||g_0 - g_1||_{h^{4,\alpha}} < \min\{\rho/C, r\}$,  we have 
$g_1(\tau) \in B_\rho(g_0(\tau)) \subset \mathcal{U}$. Then the Ricci flow starting at $g_1(\tau)$ is an extension of the flow at $g_1$ that converges to a flat metric.  
\end{proof}

As an application, consider two classes of metrics on $\bT^3$ with symmetry that were studied by Lott-Sesum \cite{LS}.  For the first class, suppose that $(\mfld^3,g)$ is obtained as a warped product $\mfld = \bS^1 \times N$, with $N = \bT^2$ with metric
\[ g = e^{2u} d\theta^2 + h, \]
where $h$ is a metric on $\bT^2$, $\theta$ is the standard coordinate on $\bS^1$ and $u \in C^{\infty}(\bT^2)$. 

The Ricci flow $g(t)$ starting at such a warped product metric preserves this symmetry and exists for all time.  Moreover, by \cite[Theorem 1.1 case (iii)]{LS}, $\lim_{t\to\infty} g(t)$ exists and converges exponentially fast to a flat metric.  Theorem B thus implies there is a small neighbourhood of the initial metric on which the flow exists for all time and still converges.  We emphasize that metrics in a $h^{2,\alpha}$-neighbourhood of a warped product metric need not be warped product metrics.  Moreover, a calculation (see \cite{LS} for details) shows that the scalar curvature of $g$ is given by
\[ \scal(g) = \scal(h) - 2 \Delta u - 2 |\nabla u|^2. \]
Choose a flat metric $(N,h)$ and highly oscillatory warping function on the torus factor.  It is possible to create a sequence of warped product metrics $g_n$ with arbitrarily large scalar curvature that leaves any (fixed) neighbourhood of a flat metric.  Yet each of these metrics possesses a neighbourhood on which every metric converges exponentially quickly to a flat metric.

For the second class of metrics, we assume that $\mfld$ fibres over $\bS^1$ with $\bT^2$ fibres.  Choose an orientation of $\bS^1$ and let $H \in SL(2,\bZ)$ be the holonomy of the torus bundle.  We further assume that $H$ has finite order.  Now regard $\mfld$ as the total space of a twisted principal $U(1)\times U(1)$ bundle with twist determined by $H$. The Ricci flow $g(t)$ starting at a metric that is invariant under the $U(1) \times U(1)$ action exists for all time, and moreover $\lim_{t\to\infty} g(t)$ exists and converges exponentially fast to a flat metric (see \cite[Theorem 1.6, case (i)]{LS}).  Once more,  convergence stability thus implies that there is a small neighbourhood of the initial metric on which the flow exists for all time, and still converges.  Calculations in \cite{CIJ} show that the scalar curvature of a class of these metrics can be made arbitrarily large.

Finally, we conclude with an example on $\bH^3$.  Consider a compact Riemannian manifold $(\mfld^3,g)$ possessing a metric of constant negative curvature, which we rescale to be $-1$.  In \cite{KnopfYoung}, the authors study a suitably normalized Ricci flow
\[ \partial_t g = - 2\Rc(g) - 4 g, \]
on a compact manifold.  An integration by parts argument using Koiso's Bochner formula allows us to bound the $L^2$ spectrum of the linearized DeTurck operator from above by a negative constant.  The results of Section 3 then show that $g$ is a stable fixed point of this flow:

\

\noindent{\bf Theorem} (Stability (Hyperbolic) \cite{KnopfYoung}).
	{\it Let $(\mfld,g)$ be a closed Riemannian manifold, where $g$ is a metric of constant negative curvature.  Then there exists $\delta > 0$ such that for all $g_0 \in B^{2+\rho}_{\delta}(g)$ the solution to the normalized flow exists for all time and converges exponentially quickly to a constant curvature hyperbolic metric.}

\

Similarly to the proof the flat case we then obtain

\begin{theorem} [Convergence Stability (Hyperbolic)] \label{Convergence to Hyperbolic}
	Let $(\mfld,g_0)$ be a closed Riemannian manifold with $g_0 \in h^{4,\alpha}(\sym)$, and let $g_0(t), \ g(0) = g_0$  be a solution of the Ricci flow that converges to a constant curvature hyperbolic metric $g_\infty$. Then there exists $\epsilon$ such that if $g_1 \in B_\epsilon(g_0) \subset h^{4,\alpha}(\sym)$, then the solution $g_1(t), g(0) = g_1$ of the Ricci flow exists for all time and converges to $g_\infty$.
\end{theorem}

We will discuss the non-compact case for hyperbolic metrics in a subsequent work.

\appendix

\section{Geometric background}
\label{Geometric background}

\
For the convenience of the reader, in this appendix  we review  notation, define appropriate function spaces and review key facts from elliptic PDE theory.  Expanded summaries can be found in \cite{Besse, Jost:Geometric Analysis}.

Let $E$ be a smooth tangent tensor bundle over a compact Riemannian manifold $(\mfld^m,g)$.  For smooth sections $u, v$, we define the $L^2$ pairing by
\[ (u, v) := \int_{\mfld} \inn{u}{v}_g \dvol_g, \]
and we set $\| u \|_{L^2}^2 := (u,u)$.  Let $C^{\infty}(E)$ denote the space of smooth sections of $E$, and define $L^2(E)$ as the completion of smooth tensor fields with respect to the $L^2$ norm.  In a similar way, for any $k \in \bN_0$ and for any $p \in [1,\infty)$, we  introduce the Sobolev space of tensor fields with $k$ covariant derivatives in $L^p$ by taking the completion of $C^{\infty}(E)$  with respect to the norm
\[ \| u \|_{L^{k,p}}^p := \sum_{i=0}^k \int_{\mfld} \| \nabla^{(i)} u \|^p_g \dvol_g, \]
where $\nabla$ denotes the covariant derivative with respect to the background metric $g$. 

We also need H\"older spaces.  Consider any finite covering of $(\mfld,g)$ by open balls $\mathscr{C} = \{(B_i, \phi_i)\}_{i=1}^N$  with charts $\phi_i: B_i \subset \mfld \to \phi(B_i) \subset \bR^m$ and let $\psi_i$ be an associated smooth  partition of unity, so that $\mathrm{supp}(\psi_i) \subset B_i$ and $\sum_{i=1}^N \psi_i = 1$.  Now if $u: \mfld \to \bR$ is a continuous function, recall that the sup-norm of $u$ is defined by
\[ \| u\|_{\infty} := \sup_{q \in \mfld} |u(q)|. \]
Note that if $u$ is entirely supported inside one coordinate chart $(B_i, \phi_i) \in \mathscr{C}$, then we may define the H\"older quotient of $u$ by simply computing the H\"older quotient of the coordinate representation of $u$ with respect to a background Euclidean metric, $\delta$; in other words,
\begin{equation} \label{holder-in-coord}
 [u]_{\alpha; B_i} := \sup_{x \neq x', x,x' \in \phi(B_i)} \frac{ | u(\phi_i^{-1}(x)) - u(\phi_i^{-1}(x')) |}{|x-x'|^{\alpha}}. 
\end{equation}

For functions not supported in a single coordinate chart, let us define a H\"older norm with respect to a covering $\mathscr{C}$ by
\[ \|u\|_{C^{0,\alpha}; \mathscr{C}} := \| u\|_{\infty} + \sum_{i=1}^N [\psi_i u]_{\alpha; B_i}, \]
and then define the space $C^{0,\alpha}(\mfld)$ as the set of continuous functions on $\mfld$ such that $\|u\|_{C^{0,\alpha};\mathscr{C}} < \infty$.  While this definition depends on the choice of cover, one may check that membership in $C^{0,\alpha}$ is independent of this choice and that these norms are all equivalent up to constants depending on the cover.  We hereafter omit the symbol $\mathscr{C}$ from the norm notation.

Regarding H\"older spaces for tensor fields, the key change needed is that in the local definition of the H\"older space for functions (see \eqref{holder-in-coord}), we replace $u(\phi_i^{-1}(x))$ by  $(\phi^{-1}_i)^* u(x)$, and then compute the H\"older quotient of the Euclidean norm of the components of $u$ in coordinates as follows:
\begin{equation} \label{holder-in-coord-tensor}
 [u]_{\alpha; B_i} := \sup_{x \neq x', x,x' \in \phi_i(B_i)} \frac{ | (\phi_i^{-1})^* u (x) - (\phi_i^{-1})^* u(x') |_{\delta}}{|x-x'|^{\alpha}}. 
\end{equation}
Using this definition for the H\"older norm on a patch $B_i$, we define the $C^{k,\alpha}$ norm on the compact manifold $\mfld$ (for any $k \in \bN_{0}$ and for any $\alpha \in (0,1)$) as follows (suppressing the dependence on $\mathscr{C}$):
\[ \|u\|_{C^{k,\alpha}} := \sum_{i = 0}^{k} \sup_{\mfld} \| \nabla^i u\|_g + \sum_{i=1}^N [\psi_i \nabla^k u]_{\alpha; B_i}. \]
We denote the space of continuous sections of $E$ for which this norm is finite by $C^{k,\alpha}(E)$.

Since the space of smooth tensor fields is not closed in the H\"older norms, in order to obtain the cleanest possible statements of our results, we also introduce the space of little H\"older continuous tensor fields.  Let $h^{k,\alpha}(E)$ denote the completion of the set of smooth tensor fields with respect to the $C^{k,\alpha}$ norm.  We note that the space $h^{k,\alpha}(E)$ is a proper closed subset of $C^{k,\alpha}(E)$ \cite{Lun95}.

We now discuss differential operators, first on functions and then on tensor fields.  Consider a linear differential operator $P$ of order $k$ acting on functions, given in local coordinates by
\[ P u = \sum_{|I|=0}^k a_I(x) \partial^I u, \]
where $I$ is a multi-index and $a_I$ is a sufficiently smooth coefficient function (the precise regularity is expounded below).  The principal symbol of $P$ is defined by 
\[ \sigma_{\xi}(P;x) := i^k \sum_{|I| = k} a_I(x) \xi^{I}. \]
The operator $P$ is defined to be {\it elliptic} if $\sigma_{\xi}(P;x) \neq 0$ for all $\xi \neq 0$.  If $P$ is a linear differential operator of order $k$ acting on sections of a tensor bundle $E$, then in terms of local coordinate-basis tensor components (with indices such as $``\ell"$ representing collective tensor indices, and with summation over repeated indices presumed) one has 
\[ (P u)_j = \sum_{|I|=0}^k {a^{\ell} _j}_I(x) \partial^I u_{\ell}, \]
with the corresponding (now tensorial) symbol
\[ \sigma_{\xi}(P;x)^{\ell}_j = i^k \sum_{|I| = k} (a_j^\ell)_I(x) \xi^{I}. \]
$P$ is elliptic if, for all $\xi \neq 0$, $\sigma_{\xi}(P;x)^\ell_j$ is an isomorphism.  $P$ is {\it strongly elliptic}
if there exists a constant $C > 0$ such that
\[ -(a_j^\ell)_I(x) \xi^I \eta^j \eta_\ell \geq C |\xi|^2 |\eta|^2, \; \forall x, \xi, \eta. \]

 If $P$ has smooth coefficients, we may regard it as a map $P:C^{\infty}(E) \to C^{\infty}(E)$.  The $L^2$-pairing for sections of  E may then be used to define the (formal) adjoint $P^*$ of $P$.  Specifically, $P^*$ is the differential operator of order $k$ such that
\[ (Pu,v) = (u, P^* v), \; \; \forall u, v \in C^{\infty}(E). \]
By definition, $P$ is formally self-adjoint if $P^*  = P$.

We use the following basic results from elliptic theory. These can be found, for example,  in Appendix A of \cite{Besse}.

\begin{theorem} \label{thm:L2-spectrum} Let $(\mfld,g)$ be a compact manifold and $E$ a tensor bundle over $\mfld$.  Let $P: C^{\infty}(E) \longrightarrow C^{\infty}(E)$ be a linear formally self-adjoint strongly elliptic geometric operator.  Then the extension of $P$ to an unbounded operator
\[ P: L^2(E) \longrightarrow L^2(E) \]
has a discrete $L^2$-spectrum consisting of a sequence of real-valued eigenvalues that diverge to $-\infty$.
\end{theorem}

In addition to this extension of the linear differential operator $P$ to $L^2(E)$, it also extends naturally to an operator mapping $C^{k+l,\alpha}(E)$ to $C^{l,\alpha}(E)$, and mapping $L^{k+l,p}(E)$ to $L^{l,p}(E)$. For $P$ acting on these H\"older and Sobolev spaces, one has the following elliptic estimates (see  Theorem 40 in the Appendix of \cite{Besse}):

\begin{theorem}[Appendix Theorem 27, \cite{Besse}]  For the linear differential operator $P$ described above, there exist positive constants $c_1$ and $c_2$ such that 
\begin{enumerate}
\item for every $u \in C^{k+l,\alpha}(E)$,
\[ \|u \|_{C^{k+l,\alpha}} \leq c_1( \| Pu \|_{C^{l,\alpha}} + \| u \|_{C^{0,\alpha}} ); \]
\item for every $u \in L^{k+l,p}(E), \; 1 < p < \infty$,
\[ \|u \|_{L^{k+l,p}} \leq c_2( \| Pu \|_{L^{l,p}} + \| u \|_{L^p} ). \]
\end{enumerate}
\end{theorem}

Although it is not stated explicitly in the reference above, the constants $c_1$ and $c_2$ appearing in the above estimates depend on upper and lower bounds for the coefficients of the operator; see \cite{GT}.


\begin{thebibliography}{10}

\bibitem {ADN 1964 35-92}Agmon, S., Douglis, A., Nirenberg, L. \emph{Estimates
	near the boundary for solutions of elliptic partial differential equations
	satisfying general boundary conditions. II.} Comm. Pure Appl. Math.
{17} (1964), 35--92.

\bibitem {Ama95}Amann, Herbert. \emph{Linear and quasilinear parabolic
problems. Vol. I. Abstract linear theory.} Monographs in Mathematics,
{89}. Birkh\"{a}user Boston, Inc., Boston, MA, 1995.

\bibitem{MCF} Angenent, S., Daskalopoulos, P., Sesum, N., \emph{Unique asymptotics of ancient convex mean curvature flow solutions}, preprint, \url{https://arxiv.org/abs/1503.01178}.

\bibitem{Besse} Besse, A. L., \emph{Einstein manifolds}. Reprint of the 1987 edition. Classics in Mathematics. Springer-Verlag, Berlin, 2008. xii+516 pp. 

\bibitem{BM} Biquard, O., Mazzeo, R., \emph{A nonlinear Poisson transform for Einstein metrics on product spaces}, J. Eur. Math. Soc. 13 (2011), no. 5, 1423--1475.

\bibitem{ChowKnopf} Chow, B., Knopf, D., \emph{The Ricci Flow: An Introduction}, Math. Surv. Mono. 110 (2004) 

\bibitem{HRF} Chow, B., Lu, P., Lei, N., \emph{Hamilton's Ricci flow.} Grad. Stud. Math. 77 (2006), 608 pp.

\bibitem{CIJ} Carfora, M., Isenberg, J., Jackson, M., \emph{Convergence of the Ricci flow for metrics with indefinite Ricci curvature.} J. Diff. Geom.  31 (1990), no. 1, 249--263. 

\bibitem{Clement}  Cl\'ement, Ph. et al. \emph{One-parameter semigroups}.  CWI Monographs, North-Holland, 1987.

\bibitem{DaPratoLunardi} DaPrato, G., Lunardi, A., \emph{Stability, instability and center manifold theorem for fully nonlinear autonomous parabolic equations in Banach Space.} Arch. Ration. Mech. Anal. {101} (1988), 115--141.

\bibitem{Eichhorn} Eichhorn, J. \emph{Global Analysis on open manifolds}.  Nova Science, New York, 2007. 

\bibitem{GT} Gilbarg, D., Trudinger, N. S. \emph{Elliptic partial differential equations of second order}.
Reprint of the 1998 edition. Classics in Mathematics. Springer-Verlag, Berlin, 2001. xiv+517 pp.

\bibitem{GIK} Guenther, C., Isenberg, J., Knopf, D.
\emph{Stability of the Ricci flow at Ricci-flat metrics}. 
Comm. Anal. Geom. 10 (2002), no. 4, 741--777. 

\bibitem{Hamilton} Hamilton, R., \emph{Three-manifolds with positive {R}icci curvature}, J. Diff. Geom. 17 (1982), 255--306.

\bibitem{IJ} Isenberg, J., Jackson, M., \emph{Ricci flow of locally homogeneous geometries on closed manifolds.} 
J. Diff. Geom. 35 (1992), no. 3, 723--741. 

\bibitem{Jost:Geometric Analysis} Jost, J. \emph{Riemannian geometry and geometric analysis}, Universitext, Springer, 2011.

\bibitem{KnopfYoung} Knopf, D., Young, A., \emph{Asymptotic stability of the cross curvature flow at a hyperbolic metric}, Proc. Amer. Math. Soc. {137} (2009), 699--709. 

\bibitem{LS} Lott, J., Sesum, N. \emph{Ricci flow on three-dimensional manifolds with symmetry}, Comm. Math. Helv. {89} (2014), no. 1, 1--32.

\bibitem{LunardiNotes} L. Lorenzi, A. Lunardi, G. Metafune, D. Pallara, \emph{Analytic semigroups and reaction-diffusion problems}. Internet seminar, 2004--2005.

\bibitem{Lun95} Lunardi, Alessandra. \emph{Analytic semigroups and optimal
regularity in parabolic problems.} Progress in nonlinear differential
equations and their applications 16, Birkh\"{a}user Boston, Boston, MA, 1995.

\bibitem{MOF} MathOverFlow question by Igor Belegradek, \emph{Does Ricci flow depend continuously on the initial metric},
\url{https://mathoverflow.net/questions/58480}

\bibitem{RFV4} Chow, B., Chu, S., Glickenstein, D., Guenther, C., Isenberg, J., Ivey, T., Knopf, D. Lu, P., Luo, F., Ni, L. \emph{The Ricci flow: Techniques and Applications. Part IV: Long Time Solutions and Related Topics. }Mathematical Surveys and
Monographs, 206, AMS, Providence, RI, 2015.

\end{thebibliography}
\end{document}